\documentclass[11pt,reqno]{amsart}
\usepackage{amsmath,amssymb,mathrsfs}
\usepackage{graphicx,cite}
\usepackage{color}
\allowdisplaybreaks[4]

\newtheorem{theorem}{Theorem}[section]
\newtheorem{lemma}{Lemma}[section]

\newtheorem{assumption}[theorem]{Assumption}
\newtheorem{remark}[theorem]{Remark}
\numberwithin{equation}{section}

\begin{document}

\title[Inverse random potential scattering]{Inverse random potential scattering for the polyharmonic wave equation using far-field patterns}

\author{Jianliang Li}
\address{Key Laboratory of Computing and Stochastic Mathematics, School of Mathematics and Statistics, Hunan Normal University, Changsha, Hunan 410081, China}
\email{lijianliang@hunnu.edu.cn, lijl@amss.ac.cn}

\author{Peijun Li}
\address{State Key Laboratory of Mathematical Sciences, Academy of Mathematics and Systems Science, Chinese Academy of Sciences, Beijing 100190, China, and School of Mathematical Sciences, University of Chinese Academy of Sciences, Beijing 100049, China}
\email{lipeijun@lsec.cc.ac.cn}  

\author{Xu Wang}
\address{State Key Laboratory of Mathematical Sciences, Academy of Mathematics and Systems Science, Chinese Academy of Sciences, Beijing 100190, China, and School of Mathematical Sciences, University of Chinese Academy of Sciences, Beijing 100049, China} 
\email{wangxu@lsec.cc.ac.cn}

\author{Guanlin Yang}
\address{State Key Laboratory of Mathematical Sciences, Academy of Mathematics and Systems Science, Chinese Academy of Sciences, Beijing 100190, China, and School of Mathematical Sciences, University of Chinese Academy of Sciences, Beijing 100049, China} 
\email{yangguanlin@lsec.cc.ac.cn}

\thanks{The work of JL is supported by the NNSF of China (12171057) and the Foundation Sciences of Hunan Province (2024JC0006). The work of PL is supported by the National Key R\&D Program of China (2024YFA1012300).
The work of XW is supported by the NNSF of China (12288201), the National Key R\&D Program of China (2024YFA1012300 and 2024YFA1015900), and the CAS Project for Young Scientists in Basic Research (YSBR-087).}



\begin{abstract}
This paper addresses the inverse scattering problem of a random potential associated with the polyharmonic wave equation in two and three dimensions. The random potential is represented as a centered complex-valued generalized microlocally isotropic Gaussian random field, where its covariance and relation operators are characterized as conventional pseudo-differential operators. Regarding the direct scattering problem, the well-posedness is established in the distributional sense for sufficiently large wavenumbers through analysis of the corresponding Lippmann--Schwinger integral equation. 
Furthermore, in the context of the inverse scattering problem, the uniqueness is attained in recovering the microlocal strengths of both the covariance and relation operators of the random potential. Notably, this is accomplished with only a single realization of the backscattering far-field patterns averaged over the high-frequency band.
\end{abstract}

\maketitle

\section{Introduction}

The polyharmonic wave equation is a higher order generalization of the classical second order wave equation. It involves the application of the polyharmonic operator, which extends the Laplacian operator to higher orders. This equation is significant in diverse areas of physics and engineering, such as in elasticity theory, fluid dynamics, and quantum mechanics. For example, the biharmonic wave equation is central to the theory of thin elastic plates, describing the deflection of a plate under load \cite{C-97}. Higher order polyharmonic equations are used in advanced models of beams and shells, where they help analyze bending and deformation of complex structures \cite{MM-IM14}. We refer to the monograph \cite{GGW-10} for a comprehensive account of boundary value problems involving polyharmonic operators. Stemming from the electrical impedance tomography problem posed by Calder\'{o}n in 1980, inverse boundary value problems for polyharmonic operators have recently received considerable attention \cite{AI-IPI19, BG-JFAA19, BKS-SIMA23, KLU14, KU-JST16}. The focus is on examining the uniqueness of lower order coefficients from boundary measurements given by the Dirichlet-to-Neumann map. These problems are one of the active research directions in inverse problem theory \cite{LL-JDE23, TH17, TS18}.

In practical applications, systems are often perturbed by uncertainties caused by multiple factors, such as the unpredictability of the environment or incomplete knowledge of the system. Consequently, it is necessary to introduce random parameters into the mathematical modeling \cite{FGPL-07, I-78}. Compared with their deterministic counterparts, both direct and inverse problems become more complex in the presence of randomness. On the one hand, random parameters are sometimes too rough to be defined point-wise, necessitating the establishment of well-posedness for direct problems with low regularity parameters. On the other hand, the randomness of the parameters makes it meaningless to reconstruct the unknown parameter path-wise; instead, it is more reasonable to determine appropriate statistics of the parameter based on the statistics of the measurements.

In this paper, we study the stochastic polyharmonic wave equation 
\begin{equation}\label{eq:model}
(-\Delta)^nu-k^{2}u+\rho u=0\quad{\rm in}~{\mathbb R}^d, 
\end{equation}
where $d=2$ or $3$, $n\geq 2$ is an integer, $k>0$ is the wavenumber, and the potential $\rho$ is assumed to be a complex-valued generalized microlocally isotropic Gaussian (GMIG) random field. The specific requirements for this assumption are detailed in subsection \ref{sec:crf}. The total wave field $u$ in \eqref{eq:model} comprises the superposition of the incident wave field $u^{\rm i}$ and the scattered wave field $u^{\rm s}$. The incident wave field $u^{\rm i}$ is assumed to be generated by a plane wave
\begin{equation}\label{uinc}
u^{\rm i}(x,\theta,\kappa)=e^{{\rm i}\kappa x\cdot \theta}, 
\end{equation}
where $\theta\in {\mathbb S}^{d-1}:=\{x\in{\mathbb R}^d: |x|=1\}$ is the incident direction and $\kappa=k^{\frac{1}{n}}>0$ denotes the modified wavenumber. It can be verified that the incident wave field satisfies $(-\Delta)^n u^{\rm i}-\kappa^{2n}u^{\rm i}=0$ in ${\mathbb R}^d$. 

To ensure the well-posedness of the scattering problem in the whole space $\mathbb R^d$, the scattered wave $u^{\rm s}$ is required to satisfy the radiation condition 
\begin{equation}\label{eq:radia}
\lim_{r\to\infty}\int_{\partial B_r}\left|\partial_{\nu}u^{\rm s}(x,\theta,\kappa)-{\rm i}\kappa u^{\rm s}(x,\theta,\kappa)\right|^2d\gamma(x)=0,
\end{equation}
where $B_r$ denotes the ball centered at the origin with radius $r$, and $\nu$ denotes the unit normal vector of the sphere $\partial B_r$. It is worth mentioning that only a single radiation condition on $u^{\rm s}$, as given in \eqref{eq:radia}, is required,  instead of $n$ conditions on $(-\Delta)^ju^{\rm s}$ with $j=0,1,\cdots,n-1$ for the $2n$th order partial differential equation \eqref{eq:model}. This requirement aligns with the radiation condition for the biharmonic wave equation in the case of $n=2$, as investigated in \cite{BH20}. Further details about the radiation condition are provided in subsection \ref{sec:radia}. The radiation condition \eqref{eq:radia} guarantees that the scattered wave $u^{\rm s}$ exhibits the asymptotic behavior 
\begin{equation*}\label{a3}
u^{\rm s}(x,\theta,\kappa)=\frac{e^{{\rm i}\kappa |x|}}{|x|^{\frac{d-1}{2}}}\left[u^{\infty}(\hat{x},\theta,\kappa)+O(|x|^{-1})\right],\quad |x|\to\infty,
\end{equation*}
where $u^{\infty}$ is the far-field pattern of the scattered field $u^{\rm s}$, and $\hat{x}:=x/|x|\in\mathbb S^{d-1}$  is referred to as the observation direction of the far-field pattern.

Based on the polyharmonic wave equation \eqref{eq:model} and the radiation condition \eqref{eq:radia}, there are two types of scattering problems to be studied. The direct scattering problem involves investigating the existence, uniqueness, and regularity of the solution $u$ in an appropriate sense, given the random potential $\rho$ and the incident wave field $u^{\rm i}$. In contrast, the inverse scattering problem aims to determine certain statistics of the unknown random potential $\rho$ from external measurements of the scattered field $u^{\rm s}$.

The inverse random potential scattering problems for second order wave equations were studied in \cite{LPS08, CHL19, LLW22a, LLM21, LLW23} under the assumption that the potential is a real-valued GMIG random field. More precisely, in the two-dimensional case, it was shown that the microlocal strength of the covariance operator of the random potential can be uniquely determined using a single realization of the near-field data associated with point sources. This uniqueness was demonstrated for the Schr\"{o}dinger equation in \cite{LPS08} and for the elastic wave equation in \cite{LLW22a}. In the three-dimensional case, the strength of the covariance operator of the random potential was shown to be uniquely recovered by a single realization of the far-field pattern associated with plane incident waves. Studies of this uniqueness were conducted for the Schr\"{o}dinger equation in \cite{CHL19, LLM21} and for the elastic wave equation in \cite{LLW23}. However, there is a lack of results concerning the inverse random potential problem based the far-field pattern generated by plane waves in two dimensions or on the near-field data associated with point sources in three dimensions. The primary challenges are twofold: (i) in the two-dimensional case, deriving an explicit form of the far-field pattern for the higher order terms in the Born series is difficult due to the complex series representation of the Hankel function, and (ii) in the three-dimensional case, the decay rate of the fundamental solution is insufficient to guarantee the convergence of the Born series. For fourth order wave equations, such as the biharmonic wave equation, the fundamental solution exhibits a higher decay rate compared to second order wave equations, including acoustic and elastic wave equations. This property facilitates a unified treatment of inverse random potential problems in both two and three dimensions using near-field data associated with point sources, provided certain regularity assumptions on the random potential are satisfied \cite{LW24}. For general higher order wave equations driven by complex-valued random noises, no results on the inverse random potential problem currently exist, to the best of our knowledge. The challenges lie in the complex structure of the random potential and the fundamental solution, as well as the absence of a unique continuation principle.

This work addresses the inverse scattering problem for the stochastic polyharmonic wave equation, with the aim of determining appropriate statistics of the random potential $\rho$ under a relaxed regularity assumption. It contains three main contributions. First, the random potential is extended to a complex-valued GMIG random field of order $(-m_1,-m_2)$ for the first time in inverse random potential problems. Its distribution is determined not only by its mean value and covariance operator of order $-m_1$, but also by its relation operator of order $-m_2$. Consequently, microlocal strengths of both the covariance and relation operators need to be reconstructed simultaneously. Second, the well-posedness of the direct scattering problem is established for sufficiently large wavenumbers in the distributional sense under a relaxed condition $m=\min\{m_1,m_2\}\in(d-2n+1,d]$ on the regularity of the random potential. This condition accommodates the white noise, i.e., $m=0$, for the case $n=2$ and $d=2$ or the case $n\ge3$. Third, it establishes the uniqueness of determining the strengths for covariance and relation operators of the random potential in the almost surely sense within a unified framework for both the two- and three-dimensional cases, based on a countable set of measurements. Specifically, to uniquely reconstruct the strengths, a single realization of the far-field patterns of the scattered field $u^{\rm s}$ generated by plane waves observed at a sequence of points with an accumulation point is required.

We introduce some general notations used in this paper. The notation $a\lesssim b$ means $a\leq Cb$ for some constant $C>0$, which may change from line to line in the proofs. Define $a\vee b:=\max\{a,b\}$ and $a\wedge b:=\min\{a,b\}$. The notation `$\mathbb P\text{-}a.s.$' represents that equations hold in the almost surely sense, i.e., with probability one. The notation  $\widehat{\phi}$ stands for the Fourier transform of a function $\phi$ defined by $\widehat{\phi}(\xi):=\int_{{\mathbb R}^d}\phi(x)e^{-{\rm i}x\cdot\xi}dx$.

The structure of the paper is as follows. Section \ref{sec:pre} introduces preliminaries on real- and complex-valued GMIG random fields. Additionally, it discusses the radiation condition of the polyharmonic wave equation. Section \ref{sec:direct} examines the well-posedness of the direct scattering problem \eqref{eq:model}--\eqref{eq:radia}. Section \ref{sec:inverse} focuses on the inverse scattering problem, which is to determine the strengths of both the covariance and relation operators for the random potential. The paper ends with final remarks in section \ref{sec:con}. 

\section{Preliminaries}\label{sec:pre}

In this section, we briefly introduce real- and complex-valued GMIG random fields. Furthermore, we justify that the radiation condition \eqref{eq:radia} is sufficient for the scattering problem of the polyharmonic wave equation \eqref{eq:model}.

\subsection{Real-valued GMIG random fields}

Let $\mathcal{D}({\mathbb R}^d)$ represent the space of test functions, defined as $C_0^{\infty}({\mathbb R}^d)$ with a convex topology, $\mathcal{D}'({\mathbb R}^d)$ denote the dual space of $\mathcal{D}({\mathbb R}^d)$, and $\langle \cdot,\cdot\rangle$ denote the dual product between $\mathcal{D}({\mathbb R}^d)$ and $\mathcal{D}'({\mathbb R}^d)$.

Let $(\Omega,\mathcal{F},{\mathbb P})$ be a complete probability space. A distribution $\rho$ is a real-valued generalized Gaussian random field if $\rho(\omega)\in\mathcal{D}'({\mathbb R}^d)$ for each $\omega\in\Omega$ and the mapping 
$
\omega\mapsto\langle \rho(\omega), \psi\rangle\in\mathbb R
$
defines a real-valued Gaussian random variable for any $\psi\in\mathcal{D}({\mathbb R}^d)$. 

The covariance operator $\mathcal C_{\rho}: \mathcal{D}({\mathbb R}^d)\to\mathcal{D}'({\mathbb R}^d)$ of the generalized Gaussian random field $\rho$ is defined by
\begin{equation*}
\langle \mathcal C_{\rho}\varphi,\psi\rangle:={\mathbb E}\left[\left(\langle \rho,\varphi\rangle-{\mathbb E}\langle \rho,\varphi\rangle\right)\left(\langle \rho,\psi\rangle-{\mathbb E}\langle \rho,\psi\rangle\right)\right]\quad\forall\,\varphi,\psi\in \mathcal{D}({\mathbb R}^d).
\end{equation*}
By the Schwartz kernel theorem, there exists a unique kernel $K_{\rho}^c\in\mathcal D'(\mathbb R^d\times\mathbb R^d)$ such that
$\langle \mathcal C_{\rho}\varphi,\psi\rangle=\langle K^c_{\rho}, \psi\otimes\varphi\rangle$ for any $\psi,\varphi\in {\mathcal D}({\mathbb R}^d).$ 
Thus, $K_\rho^c$ and $\mathcal C_\rho$ can be expressed informally as
\begin{equation*}
K^c_{\rho}(x,y)={\mathbb E}\left[(\rho(x)-\mathbb E[\rho(x)])(\rho(y)-\mathbb E[\rho(y)])\right],\quad (\mathcal C_{\rho}\varphi)(x)=\int_{{\mathbb R}^d}K^c_{\rho}(x,y)\varphi(y)dy
\end{equation*}
such that
\[
\langle\mathcal C_\rho\varphi,\psi\rangle=\int_{\mathbb R^d}\int_{\mathbb R^d}K_{\rho}^c(x,y)\varphi(y)\psi(x)dydx.
\]

A real-valued generalized Gaussian random field $\rho$ is called a generalized microlocally isotropic Gaussian (GMIG) of order $-m$ in $D$ if, in addition, $\mathcal C_{\rho}$ is a classical pseudo-differential operator whose symbol $\sigma_{\rho}^c\in {\mathcal S}^{-m}$ has the form
\begin{equation*}
\sigma_{\rho}^c(x,\xi)=a_{\rho}^c(x)|\xi|^{-m}+b_{\rho}^c(x,\xi),
\end{equation*}
where $a_{\rho}^c\in C_0^{\infty}(D)$ is referred to as the microlocal strength of $\rho$ satisfying $a_{\rho}^c\geq 0$, and $b_{\rho}^c\in\mathcal S^{-m-1}$. Here, ${\mathcal S}^{-m}$ denotes the space of symbols of order $-m$ on ${\mathbb R}^d\times {\mathbb R}^d$, defined as 
 \begin{equation*}\label{b1}
{\mathcal S}^{-m}={\mathcal S}^{-m}({\mathbb R}^d\times{\mathbb R}^d):=\left\{\sigma\in C^{\infty}({\mathbb R}^d\times{\mathbb R}^d): \big|\partial_{\xi}^{\alpha}\partial_x^{\beta}\sigma(x,\xi)\big|\leq C_{\alpha,\beta}(1+|\xi|)^{-m-|\alpha|}\right\},
\end{equation*}
where $C_{\alpha, \beta}>0$ is a constant, $\alpha$ and $\beta$ are multi-indices with $|\alpha|:=\sum\limits_{j=1}^d\alpha_j$ for $\alpha=(\alpha_1,...,\alpha_d)$. It can be verified that the relationship between the covariance kernel $K_\rho^c$ and the symbol $\sigma_\rho^c$ of the GMIG random field $\rho$ is given by 
\begin{equation*}
K^c_{\rho}(x,y)=\frac{1}{(2\pi)^d}\int_{{\mathbb R}^d}e^{{\rm i}(x-y)\cdot \xi}\sigma_{\rho}^c(x,\xi)d\xi,
\end{equation*}
or equivalently, 
\begin{equation*}
\sigma_{\rho}^c(x,\xi)=\int_{{\mathbb R}^d}K^c_{\rho}(x,y)e^{-{\rm i}(x-y)\cdot\xi}dy.
\end{equation*}
We refer to \cite{LPS08, CHL19} for more details on real-valued GMIG random fields.

\subsection{Complex-valued GMIG random fields}\label{sec:crf}

Next, we consider the complex-valued random field $\rho=\rho_1+{\rm i}\rho_2$, where $\rho_j$ for $j=1,2$ are independent real-valued GMIG random fields of orders $-m_j$ in a bounded domain $D\subset {\mathbb R}^d$. This random field $\rho$ is referred to as a complex-valued GMIG random field of order $(-m_1,-m_2)$ in $D$.

For simplicity, we assume that $\rho_j$, $j=1,2$, are centered random fields, satisfying $\mathbb E\langle\rho_j,\varphi\rangle=0$ for any $\varphi\in\mathcal D(\mathbb R^d)$. Therefore, $\rho$ is also a centered Gaussian random field, determined not only by its covariance operator $\mathcal C_{\rho}$ but also by its relation operator $\mathcal R_{\rho}$, defined as follows:
\begin{equation*}
\langle \mathcal C_{\rho}\varphi, \psi \rangle:={\mathbb E}\left[\langle \overline{\rho}, \varphi\rangle \langle \rho, \psi \rangle \right],\quad
\langle \mathcal R_{\rho}\varphi, \psi \rangle:={\mathbb E}\left[ \langle \rho, \varphi \rangle \langle \rho, \psi \rangle \right]\quad\forall\,\varphi,\psi\in{\mathcal D}({\mathbb R}^d).
\end{equation*}

Similar to the case of real-valued GMIG random fields, there exist unique kernels $K_{\rho}^c$ and $K_{\rho}^r$ in ${\mathcal D}'({\mathbb R}^d\times{\mathbb R}^d)$ such that 
$
\langle\mathcal C_{\rho}\varphi, \psi \rangle=\langle K_{\rho}^c, \psi\otimes\varphi \rangle,
$ 
$
\langle\mathcal R_{\rho}\varphi,\psi \rangle=\langle K_{\rho}^r, \psi\otimes\varphi \rangle,
$
and they can be expressed as $K_{\rho}^c(x,y)={\mathbb E}\big[\rho(x)\overline{\rho(y)}\big]$ and $K_{\rho}^r(x,y)={\mathbb E}\left[\rho(x)\rho(y)\right]$. 

Since $\rho_1$ and $\rho_2$ are independent, the covariance and relation operators of $\rho$ are still pseudo-differential operators satisfying
$\mathcal C_{\rho}=\mathcal C_{\rho_1}+\mathcal C_{\rho_2}$ and $\mathcal R_{\rho}=\mathcal C_{\rho_1}-\mathcal C_{\rho_2}$, where their symbols and kernels are given by 
\begin{eqnarray}\label{eq:sigma}
\sigma_{\rho}^c(x,\xi)=\sigma_{\rho_1}^c(x,\xi)+\sigma_{\rho_2}^c(x,\xi),\quad \sigma_{\rho}^r(x,\xi)=\sigma_{\rho_1}^c(x,\xi)-\sigma_{\rho_2}^c(x,\xi),
\end{eqnarray}
and 
\begin{eqnarray*}
K_{\rho}^c(x,y)=K^c_{\rho_1}(x,y)+K^c_{\rho_2}(x,y),\quad K_{\rho}^r(x,y)=K^c_{\rho_1}(x,y)-K^c_{\rho_2}(x,y),
\end{eqnarray*}
respectively. As in subsection 2.1, the symbols $\sigma_{\rho}^\eta$ and kernels $K^\eta_{\rho}$, where $\eta\in\{c,r\}$, satisfy the following relation:
\begin{equation}\label{eq:sigma_K}
\sigma_{\rho}^\eta(x,\xi)=\int_{{\mathbb R}^d}K^\eta_{\rho}(x,y)e^{-{\rm i}(x-y)\cdot\xi}dy.
\end{equation} 

The regularity of $\rho$ depends on that of $\rho_j$, $j=1,2$, as specified in the subsequent lemma (cf. \cite[Lemma 2.1]{LLW23}).

\begin{lemma}\label{lm:rho}
Let $\rho$ be a centered complex-valued GMIG random field of order $(-m_1,-m_2)$ in a bounded domain $D\subset {\mathbb R}^d$, where $m:=\min\{m_1,m_2\}\le d$. Then $\rho\in W^{\frac{m-d}{2}-\epsilon, p}(D)$ for any $\epsilon>0$ and $p>1$.
\end{lemma}

In this work, the random potential $\rho$ satisfies the following assumption.

\begin{assumption}\label{as:rho}
We assume that the potential $\rho=\rho_1+{\rm i}\rho_2$ is a centered complex-valued GMIG random field of order $(-m_1,-m_2)$ within a bounded domain $D\subset\mathbb R^d$, where $m_1,m_2\in(d-2n+1,d]$, and $\rho_1, \rho_2$ denote the independent real and imaginary parts of $\rho$, respectively. 
\end{assumption}

As a result, the covariance operator $\mathcal C_{\rho}$ and the relation operator $\mathcal R_{\rho}$ are classical pseudo-differential operators of order $-m$, where $m=\min\{m_1,m_2\}\in(d-2n+1,d]$, with symbols $\sigma_{\rho}^{c}$ and $\sigma_{\rho}^{r}\in {\mathcal S}^{-m}(\mathbb R^d\times\mathbb R^d)$ satisfying
\begin{equation*}
\sigma_{\rho}^{\eta}(x,\xi)=a_{\rho}^{\eta}(x)|\xi|^{-m}+b_{\rho}^{\eta}(x,\xi),
\end{equation*}
where $b_{\rho}^{\eta}\in {\mathcal S}^{-m-1}(\mathbb R^d\times\mathbb R^d)$ and $a_{\rho}^\eta, b_{\rho}^\eta(\cdot,\xi)\in C_0^\infty(D)$ for $\eta\in\{c,r\}$. Here $a_\rho^c$ and $a_\rho^r$ are called the strengths of the covariance operator $\mathcal C_\rho$ and the relation operator $\mathcal R_\rho$ associated with the random potential $\rho$, respectively. 

\begin{remark}\label{rk:example}
An example of a real-valued GMIG is given by $\sqrt{a_{\rho}^c}(-\Delta)^{-\frac{m}{4}}\dot{W}$, as discussed in \cite[Section 2.2]{LW3}.  Building on this, a complex-valued GMIG of order $(-m_1, -m_2)$ can be constructed by defining two independent random fields:
\[
\rho_1 = \sqrt{a_{\rho_1}^c}(-\Delta)^{-\frac{m_1}{4}}\dot{W}_1, \quad  \rho_2 = \sqrt{a_{\rho_2}^c}(-\Delta)^{-\frac{m_2}{4}}\dot{W}_2,
\]
where \(W_1\) and \(W_2\) are independent real-valued Wiener processes. Setting \(\rho = \rho_1 + \mathrm{i}\rho_2\), we obtain a complex-valued GMIG of the desired order.  The characterization of this GMIG is determined not only by its covariance operator but also by its relation operator.
\end{remark}

\begin{remark}\label{rk:strength}
According to \eqref{eq:sigma},  we can observe that

\begin{itemize}

\item[(i)] if $m_1<m_2$ such that $m=m_1$, then $a_{\rho}^\eta=a_{\rho_1}^c$ for $\eta\in\{c,r\}$, implying that the real part $\rho_1$ of the random potential $\rho$ has a stronger effect compared with the imaginary part $\rho_2$.

\item[(ii)] if $m_1>m_2$ such that $m=m_2$, then $a_{\rho}^c=a_{\rho_2}^c$ and $a_{\rho}^r=-a_{\rho_2}^c$, indicating that the imaginary part $\rho_2$ of the random potential $\rho$ has a stronger effect than the real part $\rho_1$. 

\item[(iii)] if $m_1=m_2=m$, then $a_{\rho}^c=a_{\rho_1}^c+a_{\rho_2}^c$ and $a_{\rho}^r=a_{\rho_1}^c-a_{\rho_2}^c$.
\end{itemize}
\end{remark}

For the inverse scattering problem, our goal is to determine the strengths $a_{\rho}^c$ and $a_{\rho}^r$ of the covariance and relation operators for the random potential $\rho$.

\subsection{The radiation condition}\label{sec:radia}

For the $n$th order polyharmonic wave equation, it seems to be necessary to impose a total of $n$ radiation conditions on the scattered field and its derivatives, as stated below:
\begin{equation}\label{eq:radia2}
\lim_{r\to\infty}\int_{\partial B_r}\left|\partial_{\nu}(-\Delta)^ju^{\rm s}(x,\theta,\kappa)-{\rm i}\kappa(-\Delta)^ju^{\rm s}(x,\theta,\kappa)\right|^2d\gamma(x)=0,
\end{equation}
for $j=0, 1, \cdots, n-1$. Below, we deduce that the radiation conditions given in \eqref{eq:radia2} are equivalent to the radiation condition \eqref{eq:radia}, which is imposed solely on the scattered field $u^{\rm s}$.

Since the potential $\rho$ is supported in $D$, it follows from \eqref{eq:model} that the scattered field $u^{\rm s}$ satisfies
$
(-\Delta)^nu^{\rm s}-\kappa^{2n}u^{\rm s}=0
$ 
in $\mathbb R^d\backslash\overline{B_R}$, 
where $B_R\subset\mathbb R^d$ is the ball centered at the origin with a sufficiently large radius $R$ such that $D\subset B_R$. It can be verified that the following operator splittings hold:
\[
(-\Delta)^n-\kappa^{2n}=\prod_{j=0}^{n-1}(-\Delta-\kappa_j^2),
\]
where 
\begin{equation}\label{eq:kj}
\kappa_j:=\kappa e^{{\rm i}\frac{j\pi}n}
\end{equation}
for $j=0,1,\cdots,n-1$ with the imaginary parts satisfying
\begin{align} \label{eq:Im_kappa}
\Im[\kappa_0]=0,\quad\Im[\kappa_j]=\kappa\sin\left(\frac{j\pi}n\right)>0,\quad j=1,\cdots,n-1.
\end{align}
Define the functions
\[
v_j:=\frac{\kappa_j^2}{n\kappa^{2n}}\Bigg(\prod_{\substack{0\le l\le n-1\\l\neq j}}\left(-\Delta-\kappa_l^2\right)\Bigg)u^{\rm s},\quad j=0,1,\cdots,n-1,
\]
which satisfy 
$
(\Delta+\kappa_j^2)v_j=0
$ in $\mathbb R^d\backslash\overline{B_R}$ for $j=0,1,\cdots,n-1$, and 
$
\sum_{j=0}^{n-1}v_j=u^{\rm s}. 
$

For $j=0,1,\cdots,n-1$, a straightforward calculation yields 
\begin{align}\label{eq:Vand}
n\kappa^{2n}v_j&=\kappa_j^2\prod_{\substack{0\le l\le n-1\\l\neq j}}\left(-\Delta-\kappa_l^2\right)u^{\rm s}\notag\\
&=\kappa_j^2\Bigg[(-\Delta)^{n-1}u^{\rm s}+\bigg(\sum_{\substack{0\le l\le n-1\\l\neq j}}(-\kappa_l^2)\bigg)(-\Delta)^{n-2}u^{\rm s}\notag\\
&\quad+\bigg(\sum_{\substack{0\le l_1, l_2\le n-1\\l_1, l_2\neq j, l_1\neq l_2}}\left(-\kappa_{l_1}^2\right)\left(-\kappa_{l_2}^2\right)\bigg)(-\Delta)^{n-3}u^{\rm s}+\cdots+\prod_{\substack{0\le l\le n-1\\l\neq j}}\left(-\kappa_l^2\right)u^{\rm s}\Bigg]\notag\\
&=\kappa_j^2\Bigg[(-\Delta)^{n-1}u^{\rm s}+\kappa_j^2(-\Delta)^{n-2}u^{\rm s}+\kappa_j^4\sum_{\substack{0\le l\le n-1\\l\neq j}}(-\Delta)^{n-3}u^{\rm s}+\cdots+\frac{\kappa^{2n}}{\kappa_j^2}u^{\rm s}\Bigg]\notag\\
&=\kappa_j^2(-\Delta)^{n-1}u^{\rm s}+\kappa_j^4(-\Delta)^{n-2}u^{\rm s}+\kappa_j^6(-\Delta)^{n-3}u^{\rm s}+\cdots+\kappa^{2n}u^{\rm s},
\end{align}
where we used the identities
\[
\sum_{l=0}^{n-1}(-\kappa_l^2)=0,~\sum_{\substack{0\leq l_1, l_2\leq n-1\\l_1\neq l_2}}\left(-\kappa_{l_1}^2\right)\left(-\kappa_{l_2}^2\right)=0,~\cdots,~\prod_{l=0}^{n-1}(-\kappa_l^2)=-\kappa^{2n}=-\kappa_j^{2n},
\]
according to the expression of the $n$th order polynomial
$
z^n-\kappa^{2n}=\prod_{l=0}^{n-1}(z-\kappa_l^2).
$

Combining \eqref{eq:Vand} for $j=0,1,\cdots, n-1$, we establish the equivalence between the sets $\{v_j\}_{j=0,\cdots,n-1}$ and $\{(-\Delta)^ju^{\rm s}\}_{j=0,\cdots,n-1}$. This equivalence arises because the distinct coefficients  $\{\kappa_j^2,\kappa_j^4,\cdots,\kappa_j^{2n}\}$ form a Vandermonde matrix, which is invertible. Hence, the radiation conditions \eqref{eq:radia2} are equivalent to the following conditions:
\[
\lim_{r\to\infty}\int_{\partial B_r}\left|\partial_{\nu}v_j(x)-{\rm i}\kappa v_j(x)\right|^2d\gamma(x)=0,\quad j=0,1,\cdots,n-1,
\]
which are automatically satisfied for $j\ge1$ due to the exponential decay of $v_j$ since $\Im[\kappa_j]>0$ for $j\ge1$, as specified in \eqref{eq:Im_kappa}. We then conclude that the condition \eqref{eq:radia2} is equivalent to the radiation condition on $v_0$, and consequently, to the radiation condition on $u^{\rm s}=\sum_{j=0}^{n-1}v_j$ as imposed in \eqref{eq:radia}.

\section{The direct scattering problem}
\label{sec:direct}

This section examines the well-posedness of the direct scattering problem \eqref{eq:model}--\eqref{eq:radia}, which is shown to have a unique solution in the distributional sense for sufficiently large wavenumbers.

\subsection{The integral operators}

First, we introduce the Green function for the two- and three-dimensional polyharmonic wave equation, upon which two integral operators are defined.

The Green function $G$ to the polyharmonic wave operator $(-\Delta)^n-\kappa^{2n}$, defined as the fundamental solution of 
$
((-\Delta)^n-\kappa^{2n})G(x,y,\kappa)=-\delta(x-y),
$
takes the form
\begin{equation}\label{eq:Green1}
G(x,y,\kappa)=-\frac{1}{n\kappa^{2n}}\sum_{j=0}^{n-1}\kappa_j^2\Phi(x,y,\kappa_j),
\end{equation}
where
\begin{equation}\label{eq:Phi}
\Phi(x,y,\kappa_j):=\left\{\begin{aligned}
&\frac{\rm i}{4}H_0^{(1)}(\kappa_j |x-y|),\quad&d=2,\\
&\frac{1}{4\pi}\frac{e^{{\rm i}\kappa_j |x-y|}}{|x-y|},\quad&d=3,
\end{aligned}\right.
\end{equation}
is the fundamental solution of the Helmholtz equation, and $\kappa_j$ is defined in \eqref{eq:kj}. Here, $H_0^{(1)}$ is the Hankel function of the first kind with order zero. 

Define two integral operators 
\begin{align}
(\mathcal H_\kappa\phi)(x):&=\int_{\mathbb R^d}G(x,y,\kappa)\phi(y)dy,\notag\\\label{eq:Koperator}
(\mathcal K_{\kappa}\phi)(x):&=\int_{{\mathbb R}^d}G(x,y,\kappa)\rho(y)\phi(y)dy,
\end{align}
where $\rho$ is the random potential. The compactness and decay property with respect to the wavenumber $\kappa$ of the operators $\mathcal H_\kappa$ and $\mathcal K_\kappa$ are given below,  and these properties are utilized in subsequent analysis.

\begin{lemma}\label{lm:H}
Let $D$ and $B$ be two bounded domains in $\mathbb{R}^d$, with $B$ having a locally Lipschitz boundary.
\begin{itemize}
\item[(i)] For $s:=s_1+s_2\in(0,2n-1)$, where $s_1,s_2\ge0$, the operator $\mathcal H_\kappa$ is bounded from $H^{-s_1}(D)$ to $H^{s_2}(B)$ and satisfies 
\begin{align*}
\|\mathcal H_\kappa\|_{\mathcal L(H^{-s_1}(D),H^{s_2}(B))}\lesssim \kappa^{s-2n+1}. 
\end{align*}
Furthermore, for any $s\in(0,2n-1)$ and $\epsilon>0$, the operator $\mathcal H_\kappa$ is bounded from $H^{-s}(D)$ to $L^\infty(B)$, satisfying 
\begin{align*}
\|\mathcal H_\kappa\|_{\mathcal L(H^{-s}(D),L^\infty(B))}\lesssim \kappa^{s-2n+1+\frac{d+\epsilon}2}.
\end{align*}
\item[(ii)] For any $p,q>1$, where $\frac1p+\frac1q=1$, and $0<\gamma<\min\{\frac{2n-1}2,\frac{2n-1}2+(\frac1q-\frac12)d\}$, the operator $\mathcal H_\kappa$ is compact from $W^{-\gamma,p}(D)$ to $W^{\gamma,q}(B)$. 
\end{itemize}
\end{lemma}

The proof of Lemma \ref{lm:H} (i) can be found in \cite[Lemma 3.1]{LLW}. The proof of Lemma \ref{lm:H} (ii) is a straightforward extension of the biharmonic case with $n=2$ provided in \cite[Lemma 2.2]{LW24}, and therefore, it is omitted here.

\begin{lemma}\label{lm:K}
Let $B$ be a bounded domain in $\mathbb{R}^d$ with a locally Lipschitz boundary.
\begin{itemize}
\item[(i)] For any $q\in(2,A)$ and $\gamma\in (\frac{d-m}{2},n-\frac{1}{2}-d(\frac12-\frac{1}{q}))$, where 
\[
A:= 
\begin{cases}
\frac{2d}{2d-m-2n+1}, &  2d-m-2n+1>0, \\ 
\infty, &   2d-m-2n+1\leq 0, 
\end{cases}
\]
the operator $\mathcal K_\kappa$ is compact from $W^{\gamma,q}(B)$ to itself and satisfies 
\begin{align*}
	\|\mathcal K_\kappa\|_{\mathcal L(W^{\gamma,q}(B))}\lesssim \kappa^{2\gamma-2n+1+d(1-\frac{2}{q})}.
\end{align*}

\item[(ii)]
For any $s\in(\frac{d-m}{2},\frac{2n-1}2)$, the operator $\mathcal K_\kappa$ is bounded from $H^s(B)$ to itself and satisfies
\begin{align*}
\|\mathcal K_\kappa\|_{\mathcal L(H^s(B))}\lesssim \kappa^{2s-2n+1}.
\end{align*}
Moreover, for any $s\in(\frac{d-m}{2},2n-1)$ and $\epsilon>0$, the operator $\mathcal K_\kappa$ is bounded from $H^{s}(B)$ to $L^\infty(B)$, satisfying 
\begin{align*}
\|\mathcal K_\kappa\|_{\mathcal L(H^{s}(B),L^\infty(B))}\lesssim \kappa^{s-2n+1+\frac{d+\epsilon}2}. 
\end{align*}
\end{itemize}
\end{lemma}

It is worth mentioning that the condition $m>d-2n+1$ in Assumption \ref{as:rho} ensures that the intervals $(2,A)$ and $(\frac{d-m}2,n-\frac12-d(\frac12-\frac1q))$ in Lemma \ref{lm:K} are not empty. The proof of Lemma \ref{lm:K} extends the case of the biharmonic operator with $n=2$ as presented in \cite[Lemma 2.3]{LW24}. It utilizes the relation $\mathcal K_\kappa\phi=\mathcal H_\kappa(\rho\phi)$ alongside the properties of $\mathcal H_\kappa$ outlined in Lemma \ref{lm:H}. For brevity, the details of the proof are omitted here.

\subsection{Well-posedness of the direct problem}

Based on the integral operator $\mathcal K_\kappa$, we can formally rewrite the scattering problem \eqref{eq:model}--\eqref{eq:radia} as the Lippmann--Schwinger integral equation
\begin{equation}\label{eq:LS}
u-\mathcal K_\kappa u=u^{\rm i}, 
\end{equation}
where $u^{\rm i}$ is the incident wave field defined in \eqref{uinc}. 

\begin{lemma}\label{lm:LS}
For any bounded domain $B\subset\mathbb R^d$ with a locally Lipschitz boundary, the Lippmann--Schwinger equation \eqref{eq:LS} admits a unique solution $u\in W^{\gamma,q}(B)$ for sufficiently large $\kappa$, where $\gamma$ and $q$ satisfy the conditions specified  in Lemma \ref{lm:K}.
\end{lemma}

\begin{proof}
By Lemma \ref{lm:K}, the operator $\mathcal K_\kappa:W^{\gamma,q}(B)\to W^{\gamma,q}(B)$ is compact. Moreover, the incident wave field $u^{\rm i}$ is smooth and bounded in $B$, which implies $u^{\rm i}\in W^{\gamma,q}(B)$. 

It suffices to show that the homogeneous equation $u-\mathcal K_\kappa u=0$ has only the trivial solution $u\equiv0$ in $W^{\gamma,q}(B)$. In fact, if $u_*\in W^{\gamma,q}(B)$ satisfies $u_*-\mathcal K_\kappa u_*=0$, then
\[
\|u_*\|_{W^{\gamma,q}(B)}\le\|\mathcal K_\kappa\|_{\mathcal L(W^{\gamma,q}(B))}\|u_*\|_{W^{\gamma,q}(B)}
\lesssim \kappa^{2\gamma-2n+1+d(1-\frac{2}{q})}\|u_*\|_{W^{\gamma,q}(B)},
\]
which implies $u_*\equiv0$ for sufficiently large wavenumbers due to the fact that $2\gamma-2n+1+d(1-\frac{2}{q})<0$, thereby completing the proof.
\end{proof}

\begin{theorem}\label{tm:direct}
For any bounded domain $B\subset\mathbb R^d$ with a locally Lipschitz boundary, the scattering problem \eqref{eq:model}--\eqref{eq:radia} admits a unique solution $u\in W^{\gamma,q}(B)$ in the distributional sense  for sufficiently large $\kappa$, where $\gamma$ and $q$ satisfy the conditions stated in Lemma \ref{lm:K}.
\end{theorem}

\begin{proof}
To establish the existence of the solution, we demonstrate that the solution $u\in W^{\gamma,q}(B)$ of \eqref{eq:LS} is also a distributional solution of \eqref{eq:model}--\eqref{eq:radia}. In fact, since $u$ satisfies
\[
u(x,\theta,\kappa)=\int_{\mathbb R^d}G(x,y,\kappa)\rho(y)u(y,\theta,\kappa)dy+u^{\rm i}(x,\theta,\kappa),
\]
we obtain for any $\phi\in\mathcal D(B)$ that
\begin{align*}
&\left\langle\left((-\Delta)^n-\kappa^{2n}\right)u+\rho u,\phi\right\rangle\\
=&\left\langle\left((-\Delta)^n-\kappa^{2n}\right)\left(\int_{\mathbb R^d}G(\cdot,y,\kappa)\rho(y)u(y,\theta,\kappa)dy+u^{\rm i}(\cdot,\theta,\kappa)\right)+\rho u,\phi\right\rangle\\
=&\int_{\mathbb R^d}\left\langle\left((-\Delta)^n-\kappa^{2n}\right)G(\cdot,y,\kappa),\phi\right\rangle\rho(y)u(y,\theta,\kappa)dy+\langle\rho u,\phi\rangle+\left\langle\left((-\Delta)^n-\kappa^{2n}\right)u^{\rm i},\phi\right\rangle\\
=&-\langle\rho u,\phi\rangle+\langle\rho u,\phi\rangle=0
\end{align*}
based on an extension argument (cf. \cite[Theorem 2 in Section 2.4.2]{RS11}), where we used the fact that $((-\Delta)^n-\kappa^{2n})u^{\rm i}=0$.

To show the uniqueness of the solution, we split the total field $u$ into the scattered field $u^{\rm s}$ and the incident field $u^{\rm i}$, i.e., $u=u^{\rm s}+u^{\rm i}$, rewriting \eqref{eq:model} as
\[
\left((-\Delta)^n-\kappa^{2n}\right)u^{\rm s}+\rho u^{\rm s}=-\rho u^{\rm i}\quad\text{in}~\mathbb R^d, 
\]
where the scattered wave $u^{\rm s}$ satisfies the radiation condition \eqref{eq:radia}. It suffices to show that the homogeneous equation
\begin{align}\label{eq:homo}
\left((-\Delta)^n-\kappa^{2n}\right)u^{\rm s}_*+\rho u^{\rm s}_*=0\quad\text{in}~\mathbb R^d
\end{align}
with the radiation condition \eqref{eq:radia} has only the trivial solution $u^{\rm s}_*\equiv0$ in $W^{\gamma,q}(B)$.

We claim that any solution of \eqref{eq:homo} is also a solution of the homogeneous Lippmann--Schwinger integral equation $u^{\rm s}_*-\mathcal K_\kappa u^{\rm s}_*=0$. Denote by $B_r$ an open ball with radius $r\gg1$. For any fixed $x\in\mathbb R^d$, the solution $u^{\rm s}_*$ of \eqref{eq:homo} satisfies
\begin{align}\label{eq:us}
&-\int_{B_r}G(x,y,\kappa)\rho(y)u^{\rm s}_*(y,\theta,\kappa)dy+u^{\rm s}_*(x,\theta,\kappa)\notag\\
&=\int_{B_r}G(x,y,\kappa)\left((-\Delta)^n-\kappa^{2n}\right)u^{\rm s}_*(y,\theta,\kappa)dy-\int_{B_r}u^{\rm s}_*(y,\theta,\kappa)\left((-\Delta)^n-\kappa^{2n}\right)G(x,y,\kappa)dy\notag\\
&=\int_{B_r}G(x,y,\kappa)(-\Delta)^nu^{\rm s}_*(y,\theta,\kappa)dy-\int_{B_r}u^{\rm s}_*(y,\theta,\kappa)(-\Delta)^nG(x,y,\kappa)dy\notag\\
&=\sum_{j=0}^{n-1}\int_{\partial B_r}\Big[\left(\partial_\nu(-\Delta)^jG(x,y,\kappa)\right)(-\Delta)^{n-1-j}u^{\rm s}_*(y,\theta,\kappa)\notag\\
&\qquad -(-\Delta)^jG(x,y,\kappa)\left(\partial_\nu(-\Delta)^{n-1-j}u^{\rm s}_*(y,\theta,\kappa)\right)\Big]d\gamma(y)\notag\\
&=\sum_{j=0}^{n-1}\int_{\partial B_r}\Big[\left(\partial_\nu(-\Delta)^jG(x,y,\kappa)-{\rm i}\kappa(-\Delta)^jG(x,y,\kappa)\right)(-\Delta)^{n-1-j}u^{\rm s}_*(y,\theta,\kappa)\notag\\
&\qquad -(-\Delta)^jG(x,y,\kappa)\left(\partial_\nu(-\Delta)^{n-1-j}u^{\rm s}_*(y,\theta,\kappa)-{\rm i}\kappa(-\Delta)^{n-1-j}u^{\rm s}_*(y,\theta,\kappa)\right)\Big]d\gamma(y).
\end{align}
Noting that both $G$ and $u^{\rm s}_*$ satisfy the radiation condition \eqref{eq:radia}, or equivalently \eqref{eq:radia2}, we can follow the same procedure as provided in \cite[Theorem 3.3]{CK13} to obtain 
\[
\lim_{r\to\infty}\int_{\partial B_r}|(-\Delta)^jf(y)|^2d\gamma(y)\lesssim1,\quad j=0,1,\cdots,n-1
\]
for $f=G(x,\cdot,\kappa)$ or $f=u^{\rm s}_*(\cdot,\theta,\kappa)$. Letting $r\to\infty$ in \eqref{eq:us}, we get
\[
-\int_{\mathbb R^d}G(x,y,\kappa)\rho(y)u^{\rm s}_*(y,\theta,\kappa)dy+u^{\rm s}_*(x,\theta,\kappa)=0,
\]
which verifies the claim. 

It follows from Lemma \ref{lm:LS} that the homogeneous Lippmann--Schwinger integral equation has only the trivial solution $u^{\rm s}_*\equiv0$, thereby completing the proof.
\end{proof}

\section{The inverse scattering problem}\label{sec:inverse}

This section is dedicated to discussing the uniqueness of the inverse scattering problem. We begin by introducing the Born sequence, which is defined through the equivalent Lippmann--Schwinger integral equation 
\eqref{eq:LS}, ensuring that the Born series converges to the exact solution. Subsequently, detailed estimates for the Born series are provided to establish the uniqueness of the inverse scattering problem by utilizing far-field patterns of the scattered wave.

\subsection{The Born sequence}

It is shown in Theorem \ref{tm:direct} that the scattering problem (\ref{eq:model})--(\ref{eq:radia}) admits a unique solution $u$, which satisfies the equivalent Lippmann--Schwinger integral equation \eqref{eq:LS}.
Define $u_0(x,\theta,\kappa):=u^{\rm i}(x,\theta,\kappa)$ and the Born sequence based on \eqref{eq:LS} as follows:
\begin{equation}\label{eq:Born}
u_j(x,\theta,\kappa):=\left(\mathcal K_{\kappa}u_{j-1}(\cdot,\theta,\kappa)\right)(x),\quad j\geq 1.
\end{equation}

Let $B\subset\mathbb R^d$ be any bounded domain with a locally Lipschitz boundary, and let $\gamma$ and $q$ satisfy the conditions specified in Lemma \ref{lm:K}. We claim that the Born series defined in \eqref{eq:Born} converges to the solution of \eqref{eq:LS} in $W^{\gamma,q}(B)$:
\[
u(x,\theta,\kappa)=\sum_{j=0}^{\infty}u_j(x,\theta,\kappa)
\]
for a sufficiently large wavenumber $\kappa$. By Lemma \ref{lm:K}, the convergence of the Born series follows directly from
\begin{align*}
\bigg\|\sum_{j=N_1}^{N_2}u_j(\cdot,\theta,\kappa)\bigg\|_{W^{\gamma,q}(B)}
&\le \sum_{j=N_1}^{N_2}\|\mathcal K_\kappa^ju_0(\cdot,\theta,\kappa)\|_{W^{\gamma,q}(B)}\\
&\lesssim \sum_{j=N_1}^{N_2}\kappa^{(2\gamma-2n+1+d(1-\frac2q))j}\|u^{\rm i}(\cdot,\theta,\kappa)\|_{W^{\gamma,q}(B)}\to0
\end{align*}
as $N_1,N_2\to\infty$ since $\gamma<n-\frac12-d(\frac12-\frac1q)$. We denote the limit of the series by $u^*(x,\theta,\kappa):=\sum_{j=0}^{\infty}u_j(x,\theta,\kappa)$, which satisfies \eqref{eq:LS} due to
\[
\mathcal K_\kappa u^*=\sum_{j=0}^{\infty}\mathcal K_\kappa u_j=\sum_{j=1}^{\infty}u_j=u^*-u_0,
\]
thus verifying the claim. The scattered wave $u^{\rm s}=u-u^{\rm i}$ can be expressed as
\begin{equation*}
u^{\rm s}(x,\theta,\kappa)=\sum_{j=1}^{\infty}u_j(x,\theta,\kappa).
\end{equation*} 

The far-field pattern of the scattered field $u^{\rm s}$ can be decomposed into the sum of the leading term $u_1^\infty$ and the residual term $b^\infty$: 
\begin{equation*}
u^{\infty}(\hat{x},\theta,\kappa)=u_1^{\infty}(\hat{x},\theta,\kappa)+b^{\infty}(\hat{x},\theta,\kappa),
\end{equation*}
where the residual term $b^\infty$ takes the form
\begin{equation}\label{eq:b}
b^{\infty}(\hat{x},\theta,\kappa)=\sum_{j=2}^{\infty}u_j^{\infty}(\hat{x},\theta,\kappa),
\end{equation}
with $u_j^{\infty}(\hat{x},\theta,\kappa)$ representing the far-field pattern of $u_j(x,\theta,\kappa)$ for $j\geq 1$.

\subsection{Analysis of the leading term}

Referring to (\ref{eq:Koperator}) and (\ref{eq:Born}), we derive the explicit expression of $u_1$ as follows:
\begin{equation}\label{eq:u1}
u_1(x,\theta,\kappa)=\int_{{\mathbb R}^d}G(x,y,\kappa)\rho(y)u_0(y,\theta,\kappa)dy.
\end{equation}
Considering that $e^{{\rm i}\kappa_j|x|}$ decays exponentially to zero as $|x|\to\infty$ if $j>0$ according to \eqref{eq:Im_kappa}, and combining it with the asymptotic behavior for large arguments of $\Phi$ given in \eqref{eq:Phi} as described in \cite[$(2.15), (3.105)$]{CK19}, we obtain the asymptotic behavior of the fundamental solution $G$:
\begin{equation}\label{eq:Ginf}
G(x,y,\kappa)=\frac{e^{{\rm i}\kappa |x|}}{|x|^{\frac{d-1}{2}}}\left[\frac{C_d}{n}\kappa^{-(2n-\frac{d+1}{2})}e^{-{\rm i}\kappa \hat{x}\cdot y}+O\left(\frac1{|x|}\right)\right],\quad |x|\to\infty, 
\end{equation}
where 
\begin{equation}\label{eq:Cd}
C_d:=\left\{
\begin{aligned}
&\frac{e^{{\rm i}\frac{\pi}4}}{\sqrt{8\pi}},\quad& d=2,\\
&\frac1{4\pi},\quad& d=3.
\end{aligned}
\right. 
\end{equation} 
Combining (\ref{eq:u1}) with (\ref{eq:Ginf}) and choosing $\theta=-\hat x$, we get 
\begin{align*}
u_1(\hat x,-\hat x,\kappa)&=\int_{\mathbb R^d}G(x,y,\kappa)\rho(y)e^{{\rm i}\kappa y\cdot(-\hat x)}dy\\
&=\frac{e^{{\rm i}\kappa |x|}}{|x|^{\frac{d-1}{2}}}\left[\frac{C_d}{n}\kappa^{-(2n-\frac{d+1}{2})}\int_{\mathbb R^d}e^{-2{\rm i}\kappa \hat{x}\cdot y}\rho(y)dy+O\left(\frac1{|x|}\right)\right],
\end{align*} 
which yields the corresponding backscattering far-field pattern of $u_1$:
\begin{equation}\label{eq:u1inf}
u_1^{\infty}(\hat{x},-\hat{x},\kappa)=\frac{C_d}{n}\kappa^{-(2n-\frac{d+1}{2})}\int_{{\mathbb R}^d}e^{-2{\rm i}\kappa \hat{x}\cdot y}\rho(y)dy.
\end{equation}

The following result illustrates the contribution of the leading term $u_1^\infty$ in determining the strengths $a_{\rho}^c$ and $a_{\rho}^r$.

\begin{theorem}\label{tm:u1}
Let the random potential $\rho$ satisfy Assumption \ref{as:rho}. 
For any fixed $\tau\geq 0$ and all $\hat{x}\in{\mathbb S}^{d-1},$ it holds that 
\begin{equation}\label{eq:u1_tm}
\begin{aligned}
\lim_{Q\to\infty}\frac{1}{Q}\int_Q^{2Q}\kappa^{m+4n-d-1}u_1^{\infty}(\hat{x},-\hat{x},\kappa+\tau)\overline{u_1^{\infty}(\hat{x},-\hat{x},\kappa)}d\kappa &=\frac{|C_d|^2}{n^22^{m}}\widehat{a^c_{\rho}}(2\tau\hat{x})\quad\mathbb P\text{-a.s.},\\
\lim_{Q\to\infty}\frac{1}{Q}\int_Q^{2Q}\kappa^{m+4n-d-1}u_1^{\infty}(\hat{x},-\hat{x},\kappa+\tau)u_1^{\infty}(-\hat{x},\hat{x},\kappa)d\kappa &= \frac{C_d^2}{n^22^{m}}\widehat{a^r_{\rho}}(2\tau\hat{x})\quad\mathbb P\text{-a.s.},
\end{aligned}
\end{equation}
where the constant $C_d$ is given in \eqref{eq:Cd}. 
\end{theorem}

\begin{proof}
For simplicity of notations, we define
$U_+(\kappa):=u_1^\infty(\hat x, -\hat x, \kappa)$, 
$U_-(\kappa):=u_1^\infty(-\hat x, \hat x, \kappa)$,
and $U_\pm(\kappa)=p_\pm(\kappa) +{\rm i} q_\pm(\kappa)$, where $p_\pm$ and $q_\pm$ are the real and imaginary parts of $U_\pm$, respectively. 

It follows from \eqref{eq:sigma_K} and (\ref{eq:u1inf}) that we have
\begin{align*}
&{\mathbb E}\left[U_+(\kappa+\tau)\overline{U_+(\kappa)} \right]\\
&=\frac{|C_d|^2}{n^2}(\kappa+\tau)^{\frac{d+1}{2}-2n}\kappa^{\frac{d+1}{2}-2n}\int_{{\mathbb R}^d}\int_{{\mathbb R}^d}e^{-2{\rm i}(\kappa+\tau)\hat{x}\cdot y}e^{2{\rm i}\kappa\hat{x}\cdot z}{\mathbb E}\left[\rho(y)\overline{\rho(z)}\right]dydz\\
&=\frac{|C_d|^2}{n^2}(\kappa+\tau)^{\frac{d+1}{2}-2n}\kappa^{\frac{d+1}{2}-2n}\int_{{\mathbb R}^d}\left[\int_{{\mathbb R}^d}K_{\rho}^c(y,z)e^{-{\rm i}2\kappa\hat{x}\cdot (y-z)}dz\right]e^{-{\rm i}2\tau\hat{x}\cdot y}dy\\
&=\frac{|C_d|^2}{n^2}(\kappa+\tau)^{\frac{d+1}{2}-2n}\kappa^{\frac{d+1}{2}-2n}\int_{{\mathbb R}^d}\sigma_{\rho}^c(y,2\kappa\hat{x})e^{-{\rm i}2\tau\hat{x}\cdot y}dy\\
&=\frac{|C_d|^2}{n^2}(\kappa+\tau)^{\frac{d+1}{2}-2n}\kappa^{\frac{d+1}{2}-2n}\left[\int_{{\mathbb R}^d}a_{\rho}^c(y)e^{-{\rm i}2\tau\hat{x}\cdot y}dy|2\kappa\hat{x}|^{-m}+\int_{{\mathbb R}^d}b_{\rho}^c(y,2\kappa\hat{x})e^{-{\rm i}2\tau\hat{x}\cdot y}dy\right]\\
&=\frac{|C_d|^2}{n^22^{m}}\left(\frac{\kappa}{\kappa+\tau}\right)^{2n-\frac{d+1}{2}}\kappa^{-(m+4n-d-1)}\widehat{a_{\rho}^c}(2\tau\hat{x})+O\left(\kappa^{-(m+4n-d)}\right), 
\end{align*}
and similarly
\begin{align*}
&{\mathbb E}\left[U_+(\kappa+\tau)U_-(\kappa) \right]\\
&=\frac{C_d^2}{n^2}(\kappa+\tau)^{\frac{d+1}{2}-2n}\kappa^{\frac{d+1}{2}-2n}\int_{{\mathbb R}^d}\int_{{\mathbb R}^d}e^{-2{\rm i}(\kappa+\tau)\hat{x}\cdot y}e^{2{\rm i}\kappa\hat{x}\cdot z}{\mathbb E}\left[\rho(y)\rho(z)\right]dydz\\
&=\frac{C_d^2}{n^2}(\kappa+\tau)^{\frac{d+1}{2}-2n}\kappa^{\frac{d+1}{2}-2n}\int_{{\mathbb R}^d}\left[\int_{{\mathbb R}^d}K_{\rho}^r(y,z)e^{-{\rm i}2\kappa\hat{x}\cdot (y-z)}dz\right]e^{-{\rm i}2\tau\hat{x}\cdot y}dy\\
&=\frac{C_d^2}{n^2}(\kappa+\tau)^{\frac{d+1}{2}-2n}\kappa^{\frac{d+1}{2}-2n}\int_{{\mathbb R}^d}\sigma_{\rho}^r(y,2\kappa\hat{x})e^{-{\rm i}2\tau\hat{x}\cdot y}dy\\
&=\frac{C_d^2}{n^22^{m}}\left(\frac{\kappa}{\kappa+\tau}\right)^{2n-\frac{d+1}{2}}\kappa^{-(m+4n-d-1)}\widehat{a_{\rho}^r}(2\tau\hat{x})+O\left(\kappa^{-(m+4n-d)}\right),
\end{align*}
where $m>d-2n+1$ such that $m+4n-d-1>0$.  Noting that
\begin{equation*}
\lim_{Q\to\infty}\frac{1}{Q}\int_Q^{2Q}\left(\frac{\kappa}{\kappa+\tau}\right)^{2n-\frac{d+1}{2}}d\kappa=1,
\end{equation*}
we obtain
\begin{equation}\label{eq:Eu1}
\begin{aligned}
\lim_{Q\to\infty}\frac{1}{Q}\int_Q^{2Q}\kappa^{m+4n-d-1}{\mathbb E}\left[U_+(\kappa+\tau)\overline{U_+(\kappa)}\right]d\kappa &= \frac{|C_d|^2}{n^22^{m}}\widehat{a^c_{\rho}}(2\tau\hat{x}),\\
\lim_{Q\to\infty}\frac{1}{Q}\int_Q^{2Q}\kappa^{m+4n-d-1}{\mathbb E}\left[U_+(\kappa+\tau)U_-(\kappa)\right]d\kappa &= \frac{C_d^2}{n^22^{m}}\widehat{a^r_{\rho}}(2\tau\hat{x}).
\end{aligned}
\end{equation}
To prove (\ref{eq:u1_tm}), together with (\ref{eq:Eu1}), it suffices to show that 
\begin{equation}\label{eq:Y}
\lim_{Q\to\infty}\frac{1}{Q}\int_{Q}^{2Q}Y^c(\hat{x},\kappa)d\kappa = 0,\quad  \lim_{Q\to\infty}\frac{1}{Q}\int_{Q}^{2Q}Y^r(\hat{x},\kappa)d\kappa = 0\quad\mathbb P\text{-a.s.,}
\end{equation}
where
\begin{align*}
& Y^c(\hat{x},\kappa):=\kappa^{m+4n-d-1}\left(U_+(\kappa+\tau)\overline{U_+(\kappa)}-{\mathbb E}\left[U_+(\kappa+\tau)\overline{U_+(\kappa)}\right]\right),\\
& Y^r(\hat{x},\kappa):=\kappa^{m+4n-d-1}\Big(U_+(\kappa+\tau)U_-(\kappa)-{\mathbb E}\big[U_+(\kappa+\tau)U_-(\kappa)\big]\Big).
\end{align*}

We obtain from a straightforward calculation that 
\begin{align}\label{eq:UV1}
U_+(\kappa+\tau)\overline{U_+(\kappa)}&=\frac{1+{\rm i}}{2}\left[p^2_+(\kappa)+p^2_+(\kappa+\tau)+q^2_+(\kappa)+q^2_+(\kappa+\tau)\right]\nonumber\\
&\quad -\frac{1}{2}\left[\left(p_+(\kappa+\tau)-p_+(\kappa)\right)^2+\left(q_+(\kappa+\tau)-q_+(\kappa)\right)^2\right]\nonumber\\
&\quad -\frac{\rm i}{2}\left[\left(p_+(\kappa+\tau)+q_+(\kappa)\right)^2+\left(q_+(\kappa+\tau)-p_+(\kappa)\right)^2\right],
\end{align}
and
\begin{align}\label{eq:UV2}
U_+(\kappa+\tau)U_-(\kappa)
&=\frac{1+{\rm i}}{2}\left[p^2_-(\kappa)+p^2_+(\kappa+\tau)+q^2_-(\kappa)+q^2_+(\kappa+\tau)\right]\nonumber\\
&\quad -\frac{1}{2}\left[\left(p_+(\kappa+\tau)-p_-(\kappa)\right)^2+\left(q_+(\kappa+\tau)+q_-(\kappa)\right)^2\right]\nonumber\\
&\quad -\frac{\rm i}{2}\left[\left(p_+(\kappa+\tau)-q_-(\kappa)\right)^2+\left(q_+(\kappa+\tau)-p_-(\kappa)\right)^2\right].
\end{align}
Denote the set $\Theta=\Theta_1\cup\Theta_2$, where 
\begin{align*}
&\Theta_1:=\{p_+(\kappa+\tau), p_+(\kappa), p_-(\kappa), q_+(\kappa+\tau), q_+(\kappa), q_-(\kappa)\}\\
&\Theta_2:=\{p_+(\kappa+\tau)-p_+(\kappa), p_+(\kappa+\tau)-p_-(\kappa), q_+(\kappa+\tau)-q_+(\kappa), \\
&\qquad\quad q_+(\kappa+\tau)+q_-(\kappa), p_+(\kappa+\tau)+q_+(\kappa), p_+(\kappa+\tau)-q_-(\kappa),\\
&\qquad\quad q_+(\kappa+\tau)-p_+(\kappa), q_+(\kappa+\tau)-p_-(\kappa)\}.
\end{align*}
By \eqref{eq:UV1} and \eqref{eq:UV2}, it is evident that both $Y^c(\hat{x},\kappa)$ and $Y^r(\hat{x},\kappa)$ are linear combinations of random fields in the set 
\[
\left\{X_{\kappa}:=\kappa^{m+4n-d-1}\left(W_{\kappa}^2-{\mathbb E}W_{\kappa}^2\right)~\big|~W_{\kappa}\in\Theta\right\}.
\] 
Therefore, to prove (\ref{eq:Y}), it is sufficient to show for any $W_{\kappa}\in\Theta$ that 
\begin{equation}\label{eq:X}
\lim_{Q\to\infty}\frac{1}{Q}\int_Q^{2Q}X_{\kappa}d\kappa = 0\quad\mathbb P\text{-a.s.}
\end{equation}

To show \eqref{eq:X}, according to Lemma \ref{lm:X} provided after this theorem, it suffices to show that there exist some constants $\mu,\alpha\ge0$, $\beta>0$, and $\kappa_*>1$ such that for any $\kappa>\kappa_*$, the following inequality holds:
\[
\left|{\mathbb E}\left[X_{\kappa}X_{\kappa+t}\right]\right|\lesssim\left(1+\frac{t}{\kappa}\right)^\alpha(1+|t-\mu|)^{-\beta}\quad\forall\,t\ge0. 
\]
Note that
\begin{align*}
\left|{\mathbb E}\left[X_{\kappa}X_{\kappa+t}\right]\right|
&=\left|\mathbb E\left[\kappa^{m+4n-d-1}(\kappa+t)^{m+4n-d-1}\left(W_{\kappa}^2-{\mathbb E}W_{\kappa}^2\right)\left(W_{\kappa+t}^2-{\mathbb E}W_{\kappa+t}^2\right)\right]\right|\\
&=2\left[{\mathbb E}\left(\kappa^{\frac{m+4n-d-1}{2}}(\kappa+t)^{\frac{m+4n-d-1}{2}}W_{\kappa}W_{\kappa+t}\right)\right]^2,
\end{align*}
where in the last step we used the result given in \cite[Lemma 3.4]{LLW23} for $W_{\kappa}\in\Theta$ being centered real-valued Gaussian random fields. Hence, we conclude that, to show \eqref{eq:X}, it suffices to show that there exist constants $\mu,\alpha\ge0$, $\beta>0$, and $\kappa_*>1$ such that 
\begin{equation}\label{eq:W}
\left[{\mathbb E}\left(\kappa^{\frac{m+4n-d-1}{2}}(\kappa+t)^{\frac{m+4n-d-1}{2}}W_{\kappa}W_{\kappa+t}\right)\right]\lesssim\left(1+\frac{t}{\kappa}\right)^\alpha(1+|t-\mu|)^{-\beta}\quad\forall\,t\ge0
\end{equation}
for any $\kappa>\kappa_*$ and $W_{\kappa}\in\Theta$.

First, we consider the case $W_{\kappa}\in\Theta_1$. It follows from the definitions of $p_\pm(\kappa)$ and 
$q_\pm(\kappa)$ that we have
\begin{align*}
p_+(\kappa_1)p_+(\kappa_2)&=\frac{1}{4}\left[U_+(\kappa_1)+\overline{U_+(\kappa_1)}\right]\left[U_+(\kappa_2)+\overline{U_+(\kappa_2)}\right],\\
q_+(\kappa_1)q_+(\kappa_2)&=-\frac{1}{4}\left[U_+(\kappa_1)-\overline{U_+(\kappa_1)}\right]\left[U_+(\kappa_2)-\overline{U_+(\kappa_2)}\right].
\end{align*}
Evidently, to show \eqref{eq:W} for $W_{\kappa}\in\Theta_1$, we require the following estimates:
\begin{equation}\label{eq:cor}
\begin{aligned}
\left|{\mathbb E}\left[U_+(\kappa_1)\overline{U_+(\kappa_2)}\right]\right| &\lesssim \kappa_1^{\frac{d+1}{2}-2n}\kappa_2^{\frac{d+1}{2}-2n-m}\left(1+|\kappa_1-\kappa_2|\right)^{-N},\\
\left|{\mathbb E}\left[U_+(\kappa_1)U_-(\kappa_2)\right]\right|& \lesssim \kappa_1^{\frac{d+1}{2}-2n}\kappa_2^{\frac{d+1}{2}-2n-m}\left(1+|\kappa_1-\kappa_2|\right)^{-N},\\
\left|{\mathbb E}\left[U_+(\kappa_1)U_+(\kappa_2)\right]\right|& \lesssim \kappa_1^{\frac{d+1}{2}-2n}\kappa_2^{\frac{d+1}{2}-2n-m}\left(1+\kappa_1+\kappa_2\right)^{-N},\\
\left|{\mathbb E}\left[U_+(\kappa_1)\overline{U_-(\kappa_2)}\right]\right| & \lesssim \kappa_1^{\frac{d+1}{2}-2n}\kappa_2^{\frac{d+1}{2}-2n-m}\left(1+\kappa_1+\kappa_2\right)^{-N},
\end{aligned}
\end{equation}
for any $\hat{x}\in{\mathbb S}^{d-1}$, $\kappa_1,\kappa_2\geq 1$, and any fixed $N\in{\mathbb N}$. The proof of these estimates is similar to that of \cite[Lemma 3.4]{LLW22b} and is omitted here.

The above estimates in \eqref{eq:cor}, along with the fact 
\[
(1+\kappa_1+\kappa_2)^{-N}\leq (1+|\kappa_1-\kappa_2|)^{-N}\quad\forall\,\kappa_1,\kappa_2\ge1
\]
for any $N\in{\mathbb N}$,  yield that 
\begin{align}\label{eq:U}
&\left|{\mathbb E}\left[\kappa^{\frac{m+4n-d-1}{2}}(\kappa+t)^{\frac{m+4n-d-1}{2}}p_+(\kappa)p_+(\kappa+t)\right]\right|\nonumber\\
&\lesssim \kappa^{\frac{m+4n-d-1}{2}}(\kappa+t)^{\frac{m+4n-d-1}{2}}\kappa^{\frac{d+1}{2}-2n}(\kappa+t)^{\frac{d+1}{2}-2n-m}(1+t)^{-N}\nonumber\\
&\le \left(1+\frac t\kappa\right)^{\frac{|m|}2}(1+t)^{-N}, 
\end{align}
and similarly
\begin{equation}\label{eq:V}
\left|{\mathbb E}\left[\kappa^{\frac{m+4n-d-1}{2}}(\kappa+t)^{\frac{m+4n-d-1}{2}}q_+(\kappa)q_+(\kappa+t)\right]\right|\lesssim \left(1+\frac t\kappa\right)^{\frac{|m|}2}(1+t)^{-N},
\end{equation}
which indicate that \eqref{eq:W} holds for $W_{\kappa}=p_+(\kappa)$ and $W_\kappa=q_+(\kappa)$.

Since (\ref{eq:U}) and (\ref{eq:V}) hold for any $\hat{x}\in{\mathbb S}^{d-1}$, replacing $\hat{x}$ with $-\hat{x}$ in (\ref{eq:U}) and (\ref{eq:V}), we conclude that (\ref{eq:W}) also holds for $W_{\kappa}=p_-(\kappa)$ and $W_\kappa=q_-(\kappa)$. 

Moreover, replacing $\kappa$ with $\kappa+\tau$ in (\ref{eq:U}) and (\ref{eq:V}) indicates that  
(\ref{eq:W}) holds for $W_{\kappa}=p_+(\kappa+\tau)$ and $W_\kappa=q_+(\kappa+\tau)$, which completes the proof of (\ref{eq:W}) for all $W_{\kappa}\in\Theta_1$.

Next, we consider the case where $W_{\kappa}\in\Theta_2$. Let us take $W_{\kappa}=p_+(\kappa+\tau)-q_-(\kappa)$ as an example, as the estimates for the others can be derived using a similar argument. Note that 
\begin{align*}
& p_+(\kappa+\tau)q_-(\kappa+t)\\
&=\frac{1}{4{\rm i}}\left[U_+(\kappa+\tau)+\overline{U_+(\kappa+\tau)}\right]\left[U_-(\kappa+t)-\overline{U_-(\kappa+t)}\right]\\
&=\frac{1}{4{\rm i}}\Big[U_+(\kappa+\tau)U_-(\kappa+t)-U_+(\kappa+\tau)\overline{U_-(\kappa+t)}\\
&\quad +\overline{U_+(\kappa+\tau)}U_-(\kappa+t)-\overline{U_+(\kappa+\tau)U_-(\kappa+t)}\Big],
\end{align*}
which, along with (\ref{eq:cor}), implies that 
\begin{equation*}
\left|{\mathbb E}\left[p_+(\kappa+\tau)q_-(\kappa+t)\right]\right|\lesssim(\kappa+\tau)^{\frac{d+1}{2}-2n}(\kappa+t)^{\frac{d+1}{2}-2n-m}(1+|t-\tau|)^{-N}.
\end{equation*}

Similarly, we can obtain that 
\begin{align*}
\left|{\mathbb E}\left[p_+(\kappa+\tau)p_+(\kappa+t+\tau)\right]\right|& \lesssim(\kappa+\tau)^{\frac{d+1}{2}-2n}(\kappa+t+\tau)^{\frac{d+1}{2}-2n-m}(1+t)^{-N},\\
\left|{\mathbb E}\left[p_+(\kappa+t+\tau)q_-(\kappa)\right]\right|& \lesssim(\kappa+t+\tau)^{\frac{d+1}{2}-2n}\kappa^{\frac{d+1}{2}-2n-m}(1+t+\tau)^{-N},\\
\left|{\mathbb E}\left[q_-(\kappa+t)q_-(\kappa)\right]\right|& \lesssim(\kappa+t)^{\frac{d+1}{2}-2n}\kappa^{\frac{d+1}{2}-2n-m}(1+t)^{-N}.
\end{align*}
Consequently, for $W_{\kappa}=p_+(\kappa+\tau)-q_-(\kappa)$, it holds
\begin{align*}
W_{\kappa}W_{\kappa+t}&=p_+(\kappa+\tau)p_+(\kappa+t+\tau)-p_+(\kappa+\tau)q_-(\kappa+t)\\
&\quad -p_+(\kappa+t+\tau)q_-(\kappa)+q_-(\kappa+t)q_-(\kappa)
\end{align*}
such that
\begin{align*}
&\left|{\mathbb E}\left[\kappa^{\frac{m+4n-d-1}{2}}(\kappa+t)^{\frac{m+4n-d-1}{2}}W_{\kappa}W_{\kappa+t}\right]\right|\\
&\lesssim \kappa^{\frac{m+4n-d-1}{2}}(\kappa+t)^{\frac{m+4n-d-1}{2}}\Big[(\kappa+\tau)^{\frac{d+1}{2}-2n}(\kappa+t+\tau)^{\frac{d+1}{2}-2n-m}(1+t)^{-N}\\
&\quad +(\kappa+\tau)^{\frac{d+1}{2}-2n}(\kappa+t)^{\frac{d+1}{2}-2n-m}(1+|t-\tau|)^{-N}\\
&\quad +(\kappa+t+\tau)^{\frac{d+1}{2}-2n}\kappa^{\frac{d+1}{2}-2n-m}(1+t+\tau)^{-N}+(\kappa+t)^{\frac{d+1}{2}-2n}\kappa^{\frac{d+1}{2}-2n-m}(1+t)^{-N}\Big]\\
&=\kappa^{\frac m2}(\kappa+t)^{-\frac m2}\left(\frac{\kappa}{\kappa+\tau}\right)^{2n-\frac{d+1}2}\left(\frac{\kappa+t}{\kappa+t+\tau}\right)^{m+2n-\frac{d+1}2}(1+t)^{-N}\\
&\quad +\kappa^{\frac m2}(\kappa+t)^{-\frac m2}\left(\frac{\kappa}{\kappa+\tau}\right)^{2n-\frac{d+1}2}(1+|t-\tau|)^{-N}\\
&\quad +\kappa^{-\frac m2}(\kappa+t)^{\frac m2}\left(\frac{\kappa+t}{\kappa+t+\tau}\right)^{2n-\frac{d+1}2}(1+t+\tau)^{-N}
+\kappa^{-\frac m2}(\kappa+t)^{\frac m2}(1+t)^{-N}\\
&\lesssim \kappa^{\frac{m}{2}}(\kappa+t)^{-\frac{m}{2}}\left[(1+t)^{-N}+(1+|t-\tau|)^{-N}\right]
+\kappa^{-\frac{m}{2}}(\kappa+t)^{\frac{m}{2}}(1+t)^{-N}\\
&\lesssim \left(1+\frac{t}{\kappa}\right)^{\frac{|m|}2}\left[(1+t)^{-N}+(1+|t-\tau|)^{-N}\right],
\end{align*}
where we utilized the facts that $2n-\frac{d+1}2\ge0$ for any $n\ge2$, and $m+2n-\frac{d+1}2>0$ for $m\in(d-2n+1,d]$ as stated in Assumption \ref{as:rho}, and that
\[
\max\left\{\left(\frac{\kappa}{\kappa+t}\right)^{\frac m2},\left(\frac{\kappa+t}{\kappa}\right)^{\frac m2}\right\}=\left(1+\frac{t}{\kappa}\right)^{\frac{|m|}2}\quad\forall\,\kappa>1,\,t\ge0.
\]
This completes the proof of (\ref{eq:W}) for $W_{\kappa}\in \Theta_2$.

Combining the above estimate, we conclude that \eqref{eq:W} holds for all $W_{\kappa}\in \Theta$, thus completing the proof.
\end{proof}

\begin{lemma}\label{lm:X}
Consider $\{X_t\}_{t \ge 0}$ as a real-valued and centered stochastic process with continuous paths. Assume that there exist constants $\mu,\alpha\ge0,$ $\beta>0,$ and $\kappa_*>1$ such that 
\begin{equation}\label{eq:cond_X}
|\mathbb E[X_{\kappa}X_{\kappa+t}]|\lesssim\left(1+\frac{t}{\kappa}\right)^\alpha(1+|t-\mu|)^{-\beta}\quad\forall\,t\ge0,~\kappa>\kappa_*.
\end{equation}
Then, it holds that  
\begin{equation}\label{eq:result}
\lim_{Q\to\infty}\frac1Q\int_{Q}^{2Q}X_{\kappa}d\kappa=0\quad\mathbb P\text{-a.s}.
\end{equation}
\end{lemma}

\begin{proof}
Without loss of generality, we assume that $\beta\in(0,1)$; otherwise, if $\beta>1$, one can always find $\beta'\in(0,1)$ such that
\[
|\mathbb E[X_{\kappa}X_{\kappa+t}]|\lesssim(1+|t-\mu|)^{-\beta}\le(1+|t-\mu|)^{-\beta'}.
\]

For $Q$ being sufficiently large such that $Q>\kappa_*\vee(2\mu)$, it follows from the condition \eqref{eq:cond_X} that
\begin{align*}
\mathbb E\left|\frac1Q\int_Q^{2Q}X_\kappa d\kappa\right|^2
&=\frac1{Q^2}\int_Q^{2Q}\int_Q^{2Q}\mathbb E[X_sX_{\kappa}]dsd\kappa\\
&\lesssim \frac1{Q^2}\int_Q^{2Q}\int_Q^{2Q}\left(1+\frac{|\kappa-s|}{\kappa}\right)^\alpha(1+||\kappa-s|-\mu|)^{-\beta}dsd\kappa,
\end{align*}
where it holds for $\kappa,s\in[Q,2Q]$ that
\[
\left(1+\frac{|\kappa-s|}{\kappa}\right)^\alpha=\left\{
\begin{aligned}
&\left(2-\frac{s}{\kappa}\right)^\alpha\le\left(\frac32\right)^\alpha<2^\alpha,\quad&\kappa\ge s,\\
&\left(\frac{s}{\kappa}\right)^\alpha\le2^\alpha, \quad&\kappa<s.
\end{aligned}
\right.
\]
Hence, we obtain 
\begin{align*}
&\mathbb E\left|\frac1Q\int_Q^{2Q}X_\kappa d\kappa\right|^2\\
&\lesssim\frac1{Q^2}\int_Q^{2Q}\int_Q^{2Q}(1+||\kappa-s|-\mu|)^{-\beta}dsd\kappa\\
&=\frac1{Q^2}\int_Q^{2Q}\bigg[\bigg(\int_{Q}^{(\kappa-\mu)\vee Q}+\int_{(\kappa-\mu)\vee Q}^{\kappa}+\int_{\kappa}^{(\kappa+\mu)\wedge 2Q}\\
&\quad+\int_{(\kappa+\mu)\wedge 2Q}^{2Q}\bigg)(1+||\kappa-s|-\mu|)^{-\beta}ds\bigg]d\kappa\\
&=\frac1{(1-\beta)Q^2}\left(\int_Q^{Q+\mu}+\int_{Q+\mu}^{2Q-\mu}+\int_{2Q-\mu}^{2Q}\right)\bigg[(1+\kappa-\mu-Q)^{1-\beta}\\
&\quad-(1+\kappa-\mu-(\kappa-\mu)\vee Q)^{1-\beta}+2(1+\mu)^{1-\beta}\\
&\quad-(1-\kappa+\mu+(\kappa-\mu)\vee Q)^{1-\beta}-(1+\kappa+\mu-(\kappa+\mu)\wedge2Q)^{1-\beta}\\
&\quad+(1-\kappa-\mu+2Q)^{1-\beta}-(1-\kappa-\mu+(\kappa+\mu)\wedge2Q)^{1-\beta}\bigg]d\kappa\\
&=\frac{2}{(1-\beta)Q^2}\left[Q(1+\mu)^{1-\beta}-2Q+2\mu+\frac{(1+Q-\mu)^{2-\beta}-(1+\mu)^{2-\beta}}{2-\beta}\right]\\
&\lesssim \frac1{Q^{\beta}}.
\end{align*}

First, we show the convergence of the result \eqref{eq:result} at discrete points $\{Q_n\}_{n\in\mathbb N}$. Define 
$Q_n:=(n+1)^k$ and $\xi_n:=\frac1{Q_n}\int_{Q_n}^{2Q_n}X_{\kappa}d\kappa$ for any $n\in\mathbb N_+$ with the constant $k$ being large enough such that $k>\frac1{\beta}$ and $Q_n>\kappa_*\vee(2\mu)$ for any $n\in\mathbb N_+$. It can be verified that 
\begin{align*}
\sum_{n=1}^\infty\mathbb E|\xi_n|^2
=\sum_{n=1}^\infty\mathbb E\left|\frac1{Q_n}\int_{Q_n}^{2Q_n}X_{\kappa}d\kappa\right|^2
\lesssim\sum_{n=1}^\infty\frac1{(n+1)^{\beta k}}<\infty,
\end{align*}
which, together with the Borel--Cantelli lemma, implies
\begin{equation}\label{eq:xi}
\lim_{n\to\infty}\xi_n=0\quad\mathbb P\text{-a.s.}
\end{equation}

Next, we show the convergence of \eqref{eq:result} for arbitrary $Q\to\infty$. Define an auxiliary random variable
\[
Y_n:=\sup_{Q_n\le Q<Q_{n+1}}\left|\frac1Q\int_Q^{2Q}X_{\kappa}d\kappa-\xi_n\right|,
\]
which satisfies
\begin{align*}
\mathbb E|Y_n|^2&=\mathbb E\left[\sup_{Q_n\le Q<Q_{n+1}}\left|\frac1Q\int_Q^{2Q}X_{\kappa}d\kappa-\frac1{Q_n}\int_{Q_n}^{2Q_n}X_{\kappa}d\kappa\right|^2\right]\\
&=\mathbb E\left[\sup_{Q_n\le Q<Q_{n+1}}\left|\frac1Q\left(\int_{Q_n}^{2Q_n}-\int_{Q_n}^Q+\int_{2Q_n}^{2Q}\right)X_\kappa d\kappa-\frac1{Q_n}\int_{Q_n}^{2Q_n}X_\kappa d\kappa\right|^2\right]\\
&\le 2\mathbb E\Bigg[\sup_{Q_n\le Q<Q_{n+1}}\left(\frac1Q-\frac1{Q_n}\right)^2\left|\int_{Q_n}^{2Q_n}X_\kappa d\kappa\right|^2\\
&\quad +\sup_{Q_n\le Q<Q_{n+1}}\left|\frac1Q\int_{Q_n}^QX_{\kappa}d\kappa\right|^2+\sup_{Q_n\le Q<Q_{n+1}}\left|\frac1Q\int_{2Q_n}^{2Q}X_{\kappa}d\kappa\right|^2\Bigg]\\
&\lesssim \left(\frac{Q_{n+1}-Q_n}{Q_n^2}\right)^2\int_{Q_n}^{2Q_n}\int_{Q_n}^{2Q_n}\mathbb E[X_sX_\kappa]dsd\kappa+\frac1{Q_n^2}\int_{Q_n}^{Q_{n+1}}\int_{Q_n}^{Q_{n+1}}\mathbb E|X_sX_\kappa|dsd\kappa\\
&\quad +\frac1{Q_n^2}\int_{2Q_n}^{2Q_{n+1}}\int_{2Q_n}^{2Q_{n+1}}\mathbb E|X_sX_\kappa|dsd\kappa\\
&\lesssim\left(\frac{Q_{n+1}-Q_n}{Q_n^2}\right)^2=\left(\frac{(n+2)^k-(n+1)^k}{(n+1)^k}\right)^2\lesssim\frac1{n^2},
\end{align*}
where we used the fact that 
$
\mathbb E|X_sX_\kappa|\le\left(\mathbb E|X_s|^2\right)^{\frac12}\left(\mathbb E|X_\kappa|^2\right)^{\frac12}\lesssim1
$
for $s,\kappa\ge Q_n>\kappa_*$ based on the condition \eqref{eq:cond_X} with $t=0$. Hence, we get
\[
\sum_{n=1}^\infty\mathbb E|Y_n|^2\lesssim\sum_{n=1}^\infty\frac1{n^2}<\infty,
\]
which indicates 
$\lim_{n\to\infty}Y_n=0$ almost surely
according to the Borel--Cantelli lemma again. The proof is completed by combining the convergence of $\{\xi_n\}_{n\in\mathbb N_+}$ given in \eqref{eq:xi}.
\end{proof}

\subsection{Analysis of the residual term}

This subsection demonstrates that the contribution of the residual term $b^{\infty}$, defined in \eqref{eq:b}, is negligible to the inverse scattering problem by utilizing the decay property of the integral operator $\mathcal K_\kappa$.

\begin{theorem}\label{tm:b}
Let the random potential $\rho$ satisfy Assumption \ref{as:rho} with the additional condition $m>\frac{4d-4n+2}3$, 
then it holds for all $\hat{x},\theta\in {\mathbb S^2}$ that
\begin{equation*}
\lim_{Q\to\infty}\frac{1}{Q}\int_Q^{2Q}\kappa^{m+4n-d-1}\left|b^{\infty}(\hat{x},\theta,\kappa)\right|^2d\kappa=0\quad\mathbb P\text{-a.s}.
\end{equation*}
\end{theorem}

\begin{proof}
From the definition of the residual in (\ref{eq:b}) and the asymptotic expansion of the fundamental solution given in (\ref{eq:Ginf}), we obtain
\begin{equation*}
b^{\infty}(\hat{x},\theta,\kappa)=\frac{C_d}{n}\kappa^{-(2n-\frac{d+1}{2})}\sum_{j=2}^{\infty}\int_{\mathbb R^d}e^{-{\rm i}\kappa\hat{x}\cdot y}\rho(y)u_{j-1}(y,\theta,\kappa)dy.
\end{equation*}
Let $\chi\in C_0^{\infty}(\mathbb R^d)$ be a cutoff function with support in a bounded domain $U$ such that $D\subset U$ and $\chi(y)=1$ if $y\in D$. For any $s\in\left(\frac{d-m}{2},\frac{2n-1}{2}\right)$, $p\geq \frac{d}{s}\vee 1$, and $p'$ satisfying $1/p+1/p'=1$, it follows from Lemma \ref{lm:rho} that $\rho\in W^{-s,p}(\mathbb R^d)$. We deduce that 
\begin{align*}
\left|b^{\infty}(\hat{x},\theta,\kappa)\right|&=\left|\frac{1}{n}C_d\kappa^{\frac{d+1}{2}-2n}\sum_{j=2}^{\infty}\int_{\mathbb R^d}\chi(y)e^{-{\rm i}\kappa\hat{x}\cdot y}\rho(y)u_{j-1}(y,\theta,\kappa)\chi(y)dy\right|\notag\\
&\lesssim\kappa^{\frac{d+1}{2}-2n}\|\rho\|_{W^{-s,p}(\mathbb R^d)}\|\chi u_0(-\hat x,\cdot,\kappa)\sum_{j=2}^{\infty}u_{j-1}(\cdot,\theta,\kappa)\|_{W^{s,p'}(\mathbb R^d)}\notag\\
&\lesssim\kappa^{\frac{d+1}{2}-2n}\|\chi u_0(-\hat x,\cdot,\kappa)\|_{H^s(\mathbb R^d)}\|\chi\sum_{j=2}^{\infty}u_{j-1}(\cdot,\theta,\kappa)\|_{H^s(\mathbb R^d)}\notag\\
&\lesssim\kappa^{\frac{d+1}{2}-2n}\kappa^s\sum_{j=2}^{\infty}\|\mathcal K_{\kappa}^{j-1}u_0(\cdot,\theta,\kappa)\|_{H^s(U)}\notag\\
&\lesssim\kappa^{\frac{d+1}{2}-2n+s}\sum_{j=2}^{\infty}\|\mathcal K_\kappa\|^{j-1}_{\mathcal L(H^s(U))}\|u_0(\cdot,\theta,\kappa)\|_{H^s(U)}\notag\\
&\lesssim\kappa^{4s-4n+\frac{d+3}2},
\end{align*}
where we used the product rule \cite[Lemma 3.6]{CHL19}, the decay property
\[
\|\mathcal K_\kappa\|_{\mathcal L(H^s(U))}\lesssim \kappa^{2s-2n+1}\quad\mathbb P\text{-a.s.}
\]
for $s\in(\frac{d-m}2,\frac{2n-1}2)$ provided in Lemma \ref{lm:K}, and the fact that $\|u_0(\cdot,\theta,\kappa)\|_{H^s(U)}\lesssim \kappa^s$ (cf. \cite{CHL19,LLW23}). A straightforward calculation yields 
\begin{eqnarray*}
\frac{1}{Q}\int_Q^{2Q}\kappa^{m+4n-d-1}\left|b^{\infty}(\hat{x},\theta,\kappa)\right|^2d\kappa\lesssim\frac{1}{Q}\int_Q^{2Q}\kappa^{m-4n+8s+2}d\kappa
\lesssim Q^{m-4n+8s+2},
\end{eqnarray*}
which converges to zero as $Q\to\infty$ if there exists some $s\in(\frac{d-m}2,\frac{2n-1}2)$ such that $m-4n+8s+2<0$, i.e., $s<\frac{4n-m-2}8$. Indeed, such a constant $s$ exists since
\[
\frac{d-m}2<\frac{4n-m-2}8
\]
provided that $m>\frac{4d-4n+2}3$, which completes the proof.
\end{proof}

\subsection{Uniqueness of the inverse problem}

With the analysis of each terms involved in the far-field pattern $u^\infty$, we are now in a position to present the main result of the paper. This result enables the unique recovery of the microlocal strengths $a_\rho^c$ and $a_\rho^r$ of the covariance and relation operators of the random potential, respectively, from a single realization of the backscattering far-field data and by calculating the Fourier transforms of $a_\rho^c$ and $a_\rho^r$.

\begin{theorem}\label{tm:main}
Let $n\ge2$, $d=2,3,$ and $\rho$ be a random potential satisfying Assumption \ref{as:rho} with 
\begin{equation}\label{eq:cond}
m\in\left(\frac{4d-4n+2}3,d\right].
\end{equation}
For any fixed $\tau\geq 0$ and all $\hat{x}\in{\mathbb S}^{d-1},$ it holds that 
\begin{align}\label{eq:main1}
\lim_{Q\to\infty}\frac{1}{Q}\int_Q^{2Q}\kappa^{m+4n-d-1}u^{\infty}(\hat{x},-\hat{x},\kappa+\tau)\overline{u^{\infty}(\hat{x},-\hat{x},\kappa)}d\kappa =&\frac{|C_d|^2}{n^22^{m}}\widehat{a^c_{\rho}}(2\tau\hat{x})\quad\mathbb P\text{-a.s.},\\\label{eq:main2}
\lim_{Q\to\infty}\frac{1}{Q}\int_Q^{2Q}\kappa^{m+4n-d-1}u^{\infty}(\hat{x},-\hat{x},\kappa+\tau)u^{\infty}(-\hat{x},\hat{x},\kappa)d\kappa =& \frac{C_d^2}{n^22^{m}}\widehat{a^r_{\rho}}(2\tau\hat{x})\quad\mathbb P\text{-a.s.}, 
\end{align}
where $C_d$ is defined in \eqref{eq:Cd}. Moreover, $a^c_{\rho}$ and $a^r_{\rho}$ can be uniquely determined by \eqref{eq:main1} and \eqref{eq:main2}, respectively, with some fixed $\tau_0>0$ and an infinite number of distinct directions $\{\hat{x}_j\}_{j\in\mathbb N}\subset{\mathbb S}^{d-1}$, or with some fixed $\hat{x}_0\in {\mathbb S}^{d-1}$ and an infinite number of distinct increments $\{\tau_j\}_{j\in\mathbb N}\subset [\tau_{\rm min},\tau_{\rm max}]$ with $0<\tau_{\rm min}<\tau_{\rm max}$.
\end{theorem}

\begin{proof}
For simplicity, we use notations $v_1=u_1^\infty$ and $v_2=b^\infty$. The far-field pattern $u^{\infty}(\hat{x},\theta,\kappa)$ can be written as 
\begin{equation*}
u^{\infty}(\hat{x},\theta,\kappa)=v_1(\hat{x}, \theta, \kappa)+v_2(\hat{x}, \theta, \kappa),
\end{equation*}
which leads to
\begin{align*}
&\frac{1}{Q}\int_Q^{2Q}\kappa^{m+4n-d-1}u^{\infty}(\hat{x},-\hat{x},\kappa+\tau)\overline{u^{\infty}(\hat{x},-\hat{x},\kappa)}d\kappa\\
&=\sum_{i,j=1}^2\frac{1}{Q}\int_Q^{2Q}\kappa^{m+4n-d-1}v_i(\hat{x},-\hat{x},\kappa+\tau)\overline{v_j(\hat{x},-\hat{x},\kappa)}d\kappa
=:\sum_{i,j=1}^2I_{i,j}
\end{align*}
and
\begin{align*}
&\frac{1}{Q}\int_Q^{2Q}\kappa^{m+4n-d-1}u^{\infty}(\hat{x},-\hat{x},\kappa+\tau)u^{\infty}(-\hat{x},\hat{x},\kappa)d\kappa\\
&=\sum_{i,j=1}^2\frac{1}{Q}\int_Q^{2Q}\kappa^{m+4n-d-1}v_i(\hat{x},-\hat{x},\kappa+\tau)v_j(-\hat{x},\hat{x},\kappa)d\kappa
=:\sum_{i,j=1}^2J_{i,j}.
\end{align*}

It follows from Theorem \ref{tm:u1} that 
\begin{align*}
\lim_{Q\to\infty}I_{1,1}=\frac{|C_d|^2}{n^22^m}\widehat{a_{\rho}^c}(2\tau\hat x),\quad 
\lim_{Q\to\infty}J_{1,1}=\frac{C_d^2}{n^22^m}\widehat{a_{\rho}^r}(2\tau\hat x)\quad\mathbb P\text{-a.s.}
\end{align*}
For the other terms $I_{i,j}$ and $J_{i,j}$ with $(i,j)\in\{(i,j):i\neq1\,\text{or}\,j\neq1\}$, we have from Theorem \ref{tm:b} that 
\begin{align*}
\lim_{Q\to\infty}I_{i,j}&\le\lim_{Q\to\infty}\Bigg\{\left[\frac{1}{Q}\int_Q^{2Q}\kappa^{m+4n-d-1}\left|v_i(\hat{x},-\hat{x},\kappa+\tau)\right|^2d\kappa\right]^{\frac{1}{2}}\\
&\quad \times\left[\frac{1}{Q}\int_Q^{2Q}\kappa^{m+4n-d-1}\left|v_j(\hat{x},-\hat{x},\kappa)\right|^2d\kappa\right]^{\frac{1}{2}}\Bigg\}
=0,\\
\lim_{Q\to\infty}J_{i,j}&\le \lim_{Q\to\infty}\Bigg\{\left[\frac{1}{Q}\int_Q^{2Q}\kappa^{m+4n-d-1}\left|v_i(\hat{x},-\hat{x},\kappa+\tau)\right|^2d\kappa\right]^{\frac{1}{2}}\\
&\quad \times\left[\frac{1}{Q}\int_Q^{2Q}\kappa^{m+4n-d-1}\left|v_j(-\hat{x},\hat{x},\kappa)\right|^2d\kappa\right]^{\frac{1}{2}}\Bigg\}
=0. 
\end{align*}
We then conclude that 
\begin{equation*}
\begin{aligned}
\lim_{Q\to\infty}\frac{1}{Q}\int_Q^{2Q}\kappa^{m+4n-d-1}u^{\infty}(\hat{x},-\hat{x},\kappa+\tau)\overline{u_1^{\infty}(\hat{x},-\hat{x},\kappa)}d\kappa &=\frac{|C_d|^2}{n^22^{m}}\widehat{a^c_{\rho}}(2\tau\hat{x})\quad\mathbb P\text{-a.s.},\\
\lim_{Q\to\infty}\frac{1}{Q}\int_Q^{2Q}\kappa^{m+4n-d-1}u^{\infty}(\hat{x},-\hat{x},\kappa+\tau)u_1^{\infty}(-\hat{x},\hat{x},\kappa)d\kappa &= \frac{C_d^2}{n^22^{m}}\widehat{a^r_{\rho}}(2\tau\hat{x})\quad\mathbb P\text{-a.s.}
\end{aligned}
\end{equation*}

Since $a_{\rho}^c,a_{\rho}^r\in C_0^{\infty}(D)$, their Fourier transforms $\widehat{a_{\rho}^c}$ and $\widehat{a_{\rho}^r}$ are analytic and can thus be uniquely determined by their values at a countable sequence $\{\xi_j\}_{j\in\mathbb N}$ with an accumulation point (cf. \cite[Chapter 4, Section 3.2]{A78}). 
Moreover, due to the compactness of the finite dimensional unit sphere ${\mathbb S}^{d-1}$ and the closed interval $[\tau_{\min},\tau_{\max}]$, the sequence $\{2\tau\hat{x}\}$ has an accumulation point for some fixed $\tau_0>0$ and an infinite number of distinct directions $\{\hat{x}_j\}_{j\in\mathbb N}\subset{\mathbb S}^{d-1}$, or for some fixed $\hat{x}_0\in{\mathbb S}^{d-1}$ and an infinite number of increments $\{\tau_j\}_{j\in\mathbb N}\subset [\tau_{\min},\tau_{\max}]$. This completes the proof of the uniqueness for determining the strengths $a_\rho^c$ and $a_{\rho}^r$.
\end{proof}

\begin{remark}\label{rk:cond}
It is worth noting that for $n=1$, the condition \eqref{eq:cond} turns to be $m\in\emptyset$ for $d=2,3$. 
It coincides with the acoustic wave equation (cf. \cite{CHL19,LPS08}) and the elastic wave equation (cf. \cite{LLW22a,LLW23}) that (i) for the two-dimensional case, the near-field data observed in an open domain will be used instead of the far-field data to uniquely determine the strength of the random potential; (ii) for the three-dimensional case, the second term $u_2^\infty$ defined in the Born sequence need to be estimated separately from the residual \eqref{eq:b} to get a sharper condition on $m$.
\end{remark}

\begin{remark}
If, in particular, $m_1=m_2=m$, the strengths $a_{\rho_1}^c$ and $a_{\rho_2}^c$ can be reconstructed for the real and imaginary parts of the random potential, respectively, as indicated in Remark \ref{rk:strength}.
\end{remark}

\section{Conclusion}\label{sec:con}

In this paper, we have studied the well-posedness of the direct scattering problem and the uniqueness of the inverse random potential scattering problem for the stochastic polyharmonic wave equation in both two and three dimensions. Here, the random potential is assumed to be a centered and complex-valued GMIG random field. We have demonstrated that the direct scattering problem admits a unique solution in the distributional sense for sufficiently large wavenumbers. Additionally, we have shown that a single realization of the far-field patterns is sufficient to uniquely determine the microlocal strengths of both the covariance and relation operators of the random potential.

The present paper addresses the two- and three-dimensional polyharmonic wave equations within a unified framework. This is facilitated by two key observations: (i) the far-field pattern of the first term $u_1$ in the Born series can be uniformly treated, as the fundamental solutions for both cases exhibit a unified asymptotic expansion, and (ii) the far-field patterns of the other terms $u_j$, $j\ge2$, can be considered as higher order terms whose contribution to the reconstruction is negligible under a more demanding condition \eqref{eq:cond} on the parameter $m$ than the one in the direct scattering problem.

To demonstrate the well-posedness of the direct scattering problem, the wavenumber is assumed to be sufficiently large to ensure that the operator $\mathcal{K}_\kappa$ is a contraction map as it appears in the Lippmann--Schwinger equation. It is interesting to prove that the direct scattering problem admits a unique solution for any wavenumber. This requires new techniques to demonstrate that unique continuation holds for any wavenumber. In addition, it is unclear whether the condition \eqref{eq:cond} for the inverse random potential scattering problem is optimal. As mentioned in Remark \ref{rk:cond}, the condition on $m$ could potentially be further weakened if more refined estimates were available for the other terms $u_j$, $j\ge2$, instead of treating them uniformly as a residual. However, obtaining refined estimates for these terms is intricate, as the random potential $\rho$ becomes much less regular if $m$ is smaller than assumed in Theorem \ref{tm:main}.

A potential approach to obtaining a sharper condition on $m$ is to utilize other types of measurements, such as the near-field data excited by point sources. In this scenario, the two- and three-dimensional wave equations are typically treated separately due to the distinct expressions of the fundamental solution. Moreover, the singularity of the fundamental solution present in the near-field data, as well as the incident wave, will make both the direct and inverse scattering problems more complex. Progress on these aspects will be reported elsewhere in the future.

\end{document}